%% file: main-pearl.tex
\documentclass[ a4paper
              , UKenglish
              , cleveref
              , autoref
              , thm-restate
              ]{lipics-v2021}

\author{Guido Boccali}{Università di Torino, Torino, Italy}{guidoboccali@gmail.com}{}{}
\author{Andrea Laretto}{Tallinn University of Technology, Tallinn, Estonia}{anlare@ttu.ee}{https://orcid.org/0000-0002-1825-0097}{}
\author{Fosco Loregian}{Tallinn University of Technology, Tallinn, Estonia}{folore@ttu.ee}{https://orcid.org/0000-0003-3052-465F}{}
\author{Stefano Luneia}{Università di Bologna, Bologna, Italy}{stefano.luneia@gmail.com}{}{}
\authorrunning{Boccali, Laretto, Loregian, Luneia}
\acknowledgements{\emph{\`A René, parce qu'il faut ruser pour te lire}.}
\title{Completeness for categories of generalized automata}
\date{\today}
\Copyright{Guido Boccali and Andrea Laretto and Fosco Loregian and Stefano Luneia}
\funding{F. Loregian was supported by the ESF funded Estonian IT Academy research measure (project 2014-2020.4.05.19-0001).}

\usepackage{lipics-fouche}
\input{cmds}

\def\uno{$\left\{\vcenter{\xymatrix@R=6mm@C=4mm{
	A\ar@{.>}[d]_\varphi & FA\ar[l]_a \ar@{.>}[d]^{F\varphi} \\
	E & FE\ar[l]_d
	}\xymatrix@R=6mm@C=4mm{
	A\ar@{.>}[d]_\varphi\ar[r]^-u & O_\infty \ar[r]^{s_T} & O \ar@{=}[d]\\
	E \ar[ur]_{u_E}\ar[rr]_s && O
	}}\right\}$}
\def\tre{$\left\{\vcenter{\xymatrix@R=6mm@C=6mm{
	& FA \ar@{->}[d]_{a} \ar@{.>}[r]^{F\varphi} \ar@{->}[ld]_{Fu} & FE \ar@{->}[d]^{d} \\
	FO_\infty \ar@{->}[rd]_{d_T} & A \ar@{.>}[r]_{\varphi} \ar@{->}[d]_{u} & E \ar@{->}[ld] \ar@{->}[d]^{s} \\
	& O_\infty \ar@{->}[r]_{s_T} & O
	}}\right\}$}

\input{boilerplate}

\definecolor{refkey}{HTML}{ADD8E6}
\definecolor{labelkey}{HTML}{ADD8E6}

\nolinenumbers
\numberwithin{theorem}{section}
\numberwithin{definition}{section}
\numberwithin{proposition}{section}
\numberwithin{remark}{section}
\numberwithin{example}{section}
\numberwithin{lemma}{section}
\begin{document}
\maketitle

\def\xmly{F\emdash\mly} 
\def\xmoo{F\emdash\mre} 
\NewDocumentCommand{\xmach}{m O{O} m m}{\vxy{#1 & \ar[l]_-{#4} F #1 \ar[r]^-{#3} & #2}} 

\begin{abstract}
We present a slick proof of completeness and cocompleteness for categories of $F$-\emph{automata}, where the span of maps $E\leftarrow E\otimes I \to O$ that usually defines a deterministic automaton of input $I$ and output $O$ in a monoidal category $(\clK,\otimes)$ is replaced by a span $E\leftarrow F E \to O$ for a generic endofunctor $F : \clK\to \clK$ of a generic category $\clK$: these automata exist in their `Mealy' and `Moore' version and form categories $\xmly$ and $\xmoo$; such categories can be presented as strict 2-pullbacks in $\Cat$ and whenever $F$ is a left adjoint, both $\xmly$ and $\xmoo$ admit all limits and colimits that $\clK$ admits.
We mechanize our main results using the proof assistant Agda and the library \href{https://github.com/agda/agda-categories}{\texttt{agda-categories}}.
\end{abstract}
\section{Introduction}\label{sec:intro}

%


One of the most direct representations of \emph{deterministic automata} in the categorical settings consists (cf. \cite{adam-trnk:automata,2009,Ehrig}) of a span of morphisms $E\leftarrow E\times I \to O$, where the left leg provides a notion of \emph{next states} of the automaton given a current state $E$ and an input $I$, and the right leg provides an output $O$ given the same data. According to whether the output morphism depends on both the current state and an input or just on the state, one can then talk about classes of \emph{Mealy} and \emph{Moore automata}, respectively. This perspective of `automata in a category' naturally captures the idea that morphisms of a category can be interpreted as a general abstraction of processes/sequential operations.

The above notion of deterministic automata carries over to any monoidal category, on which the various classical notions of automata, e.g., minimization, bisimulation, powerset construction, can be equivalently reconstructed and studied as in the monograph \cite{Ehrig}.

In \cite{adam-trnk:automata, Guitart1980}, automata are generalized to the case in which, instead of taking spans from the monoidal product of states and inputs $E \otimes I$, one considers spans $E\leftarrow F E \to O$ for a generic endofunctor $F : \clK\to \clK$, providing an abstraction for the ambient structure that allows the automata to advance to the `next' state and give an output.



A general theorem asserting that the category of Mealy and Moore automata $\Mly_\clK(I,O)$, $\Mre_\clK(I,O)$ in a monoidal category $(\clK,\otimes)$ are complete and cocomplete whenever $\clK$ is itself complete and cocomplete can be obtained with little effort, cf. \cite[Ch. 11]{Ehrig}, but the proof given therein is a bit ad-hoc, and provides no intuition for why finite products and terminal objects tend to be so complicated.

With just a little bit more category-theoretic technology, some general considerations can be made about the shape of limits in such settings: colimits and connected limits can be computed as they are computed in $\clK$ (as a consequence of the fact that a certain functor \emph{creates} them, cf. \cite{McL}), whereas products (and in particular the empty product, the terminal object) have dramatically different shapes than those provided in $\clK$. The profound reason why this is the case comes from the fact that such a terminal object (which we refer to $O_\infty$) coincides with the terminal coalgebra of a specific endofunctor, which, depending on the case of Moore and Mealy automata, is given by $A \mapsto O \times R A$ and $A \mapsto R O \times R A$ respectively. The complicated shape of the terminal object $O_\infty$ in $\Mly_\clK(I,O)$ is then explained by Adamek's theorem, which presents the terminal object $O_\infty$ as inverse limit in $\clK$.

In this paper, we show that under the same assumption of completeness of the underlying category $K$, the completeness of $F$-automata can be obtained by requiring that the endofunctor $F$ admits a right adjoint $R$. The proof we provide follows a slick argument proving the existence of (co)limits by fitting each $\Mly_\clK(I,O)$ and $\Mre_\clK(I,O)$ into a strict 2-pullback in $\Cat$, and deriving the result from stability properties of limit-creating functors.

\medskip
\textbf{Outline of the paper.}
The present short note develops as follows:
\begin{itemize}
\item first (\autoref{sec:automata}) we introduce the language we will employ and the structures we will study:\footnote{An almost identical introductory short section appears in \cite{noi:bicategories}, of which the present note is a parallel submission --although related, the two manuscripts are essentially independent, and the purpose of this repetition is the desire for self-containment.} categories of automata valued in a monoidal category $(\clK,\otimes)$ (in two flavours: `Mealy' machines, where one considers spans $E\leftarrow E\otimes I\to O$, and `Moore', where instead one consider pairs $E\leftarrow E\otimes I,E\to O$) and of $F$-automata, where $F$ is an endofunctor of $\clK$ (possibly with no monoidal structure). `Mealy' automata are known as `deterministic automata' in today's parlance, but since we need to distinguish between the two kinds of diagram from time to time, we stick to an older terminology.
	\item Then (\autoref{limcolim}), to establish the presence of co/limits of shape $\clJ$ in categories of $F$-automata, under the two assumptions that $F:\clK\to\clK$ is a left adjoint in an adjunction $F\adjunct{\epsilon}{\eta} R$, and that co/limits of shape $\clJ$ exist in the base category $\clK$.
	\item Last (\autoref{roba_naude}), to address the generalisation to $F$-machines of the `behaviour as an adjunction' perspective expounded in \cite{2d281e1e3e0b525128f55519cf8a03fa52ce6252,NAUDE1979277}.
\end{itemize}
Similarly to the situation for Mealy/Moore machines, where $F=\firstblank\otimes I$, \emph{discrete} limits if $\xmly$ and $\xmoo$ exist but tend to have a shape that is dramatically different than the one in $\clK$.

A number of examples of endofunctors $F$ that satisfy the previous assumption come from considering $F$ as the (underlying endofunctor of the) comonad $LG$ of an adjunction $L\dashv G\dashv U$, since in that case $LG\dashv UG$: the shape-flat and flat-sharp adjunctions of a cohesive topos \cite{lawvere1986categories,lawvere2007axiomatic}, or the base-change adjunction $\Lan_f\dashv f^*\dashv \Ran_f$ for a morphism of rings, or more generally, $G$-modules in representation theory, any essential geometric morphism, or any topological functor $V : \clE \to \clB$ \cite[Prop. 7.3.7]{Bor2} with its fully faithful left and right adjoints $L\dashv V\dashv R$ gives rise to a comodality $LV$, left adjoint to a modality $RV$.

The results we get are not particularly surprising; we have not, however, been able to trace a reference addressing the co/completeness properties of $\xmly,\xmoo$ nor an analogue for the `behaviour as an adjunction' theorems expounded in \cite{2d281e1e3e0b525128f55519cf8a03fa52ce6252,NAUDE1979277}; in the case $F=\firstblank\otimes I$ co/completeness results follows from unwieldy ad-hoc arguments (cf. \cite[Ch. 11]{Ehrig}), whereas in \autoref{limcolim} we provide a clean, synthetic way to derive both results from general principles, starting by describing $\xmly$ and $\xmoo$ as suitable pullbacks in $\Cat$, in \autoref{machines_are_pb}.

We provide a mechanisation of our main results using the proof assistant Agda and the library \href{https://github.com/agda/agda-categories}{\texttt{agda-categories}}: we will add a small Agda logo \specialagdasymbol\ next to the beginning of a definition or statement whenever it is accompanied by Agda code, pointing directly to the formalisation files. The full development is available at \url{https://github.com/iwilare/categorical-automata}.

\section{Automata and \torpdf{$F$}-automata}\label{sec:automata}
The only purpose of this short section is to fix notation; classical comprehensive references for this material are \cite{adam-trnk:automata,Ehrig}; in particular, \cite[Ch. III]{adam-trnk:automata} is entirely devoted to the study of what here are called $F$-Moore automata, possibly equipped with an `initialization' morphism.  
\subsection{Mealy and Moore automata}
For the entire subsection, we fix a monoidal category $(\clK,\otimes,1)$.
\begin{definition}[Mealy machine]\agdasymbol{Mealy}\label{mmach_1cells}
	A \emph{Mealy machine} in $\clK$ of input object $I$ and output object $O$ consists of a triple $(E,d,s)$ where $E$ is an object of $\clK$ and $d,s$ are morphisms in a span
	\[\fke := \left(\mealy{E}{d}{s}\right)\label{mmach_eq}\]
\end{definition}
\begin{remark}[The category of Mealy machines]\label{mmach_2cells}
	Mealy machines of fixed input and output $I,O$ form a category, if we define a \emph{morphism of Mealy machines} $f : (E,d,s)\to (T, d', s')$ as a morphism $f : E\to T$ in $\clK$ such that
	\[\vxy{
		E\ar[d]_f & \ar[l]_d E\otimes I \ar[d]^{f\otimes I}\ar[r]^s& O\ar@{=}[d]\\
		T & \ar[l]_-{d'} T\otimes I \ar[r]_{s'}& O
		}
	\]
	Clearly, composition and identities are performed in $\clK$.
\end{remark}
The category of Mealy machines of input and output $I,O$ is denoted as $\Mly_\clK(I,O)$.
\begin{definition}[Moore machine]\agdasymbol{Moore}\label{moore_1cells}
	A \emph{Moore machine} in $\clK$ of input object $I$ and output object $O$ is a diagram
	\[\fkm := \left(\moore{E}{d}{s}\right)\label{momach_eq}\]
\end{definition}
\begin{remark}[The category of Moore machines]\label{moore_2cells}
	Moore machines of fixed input and output $I,O$ form a category, if we define a \emph{morphism of Moore machines} $f : (E,d,s)\to (T, d', s')$ as a morphism $f : E\to T$ in $\clK$ such that
	\[\vxy{
		E\ar[d]_f & \ar[l]_d E\otimes I \ar[d]^{f\otimes I}&E\ar[d]_f \ar[r]^s& O\ar@{=}[d]\\
		T & \ar[l]_-{d'} T\otimes I &T\ar[r]_{s'}& O
		}
	\]
\end{remark}
\subsection{$F$-Mealy and $F$-Moore automata}
The notion of \emph{$F$-machine} arises by replacing the tensor $E\otimes I$ in \eqref{mmach_eq} with the action $FE$ of a generic endofunctor $F:\clK\to\clK$ on an object $E\in\clK$, in such a way that a Mealy/Moore machine is just a $(\firstblank\otimes I)$-Mealy/Moore machine; cf. \cite[ff. 2.1.3°]{Guitart1980}, or Chapter III of the monograph \cite{adam-trnk:automata}. This natural idea acts as an abstraction for the structure that allows the machine to advance to the `next' state and give an output, and it leads to the following two definitions (where we do \emph{not} require $\clK$ to be monoidal).
\begin{definition}[$F$-Mealy machine]\agdasymbol{FMealy}\label{xmealy}
	Let $O\in\clK$ be a fixed object. The objects of the category $\xmly_{/O}$ (or simply $\xmly$ when the object $O$ is implicitly clear) of \emph{$F$-Mealy machines of output $O$} are the triples $(E,d,s)$ where $E\in\clK$ is an object and $s,d$ are morphisms in $\clK$ that fit in the span
	\[\xmach{E}{s}{d}\]
	A \emph{morphism} of $F$-Mealy machines $f : (E,d,s)\to (T,d',s')$ consists of a morphism $f : E\to T$ in $\clK$ such that
	\[\vxy{
		E\ar[d]_f & \ar[l]_d FE \ar[d]^{Ff}\ar[r]^s& O\ar@{=}[d]\\
		T & \ar[l]_-{d'} FT \ar[r]_{s'}& O
		}
	\]
\end{definition}
Unsurprisingly, we can generalise in the same fashion \autoref{moore_1cells} to the case of a generic endofunctor $F:\clK\to\clK$.
\begin{definition}[$F$-Moore machine]\agdasymbol{FMoore}\label{xmoore}
	Let $O\in\clK$ be a fixed object. The objects of the category $\xmoo_{/O}$ of \emph{$F$-Moore machines of output $O$} are the triples $(E,d,s)$ where $E\in\clK$ is an object and $s,d$ are a pair of morphisms in $\clK$
	\[\vxy{E & \ar[l]_-{d} FE \:\: ; \:\: E \ar[r]^-{s} & O}\]
	A \emph{morphism} of $F$-Moore machines $f : (E,d,s)\to (T,d',s')$ consists of a morphism $f : E\to T$ in $\clK$ such that
	\[\vxy{
		E\ar[d]_f & \ar[l]_d FE \ar[d]^{Ff}&E\ar[d]_f \ar[r]^s& O\ar@{=}[d]\\
		T & \ar[l]_-{d'} FT &T\ar[r]_{s'}& O
		}
	\]
\end{definition}


\begin{remark}[Interdefinability of notions of machine]
	All the concepts of machine introduced so far are interdefinable, provided we allow the monoidal base $\clK$ to change (cf. \cite[ff. Proposition 30]{Guitart1980}): a Mealy machine is, obviously, an $F$-machine where $F: \clK\to \clK$ is the functor $\firstblank\otimes I : E\mapsto E\otimes I$; an $F$-machine consists of a Mealy machine in a category of endofunctors: in fact, $F$-machines are precisely the Mealy machines of the form $E \leftarrow F\circ E \to O$, where $E,O$ are constant endofunctors on objects of $\clK$ and $F$ is the input object: more precisely, the category of $F$-machines is contained in the category $\Mly_{([\clK,\clK],\circ)}(F,c_O)$, where $c_O$ is the constant functor on $O\in\clK$, as the subcategory of those triples $(E,d,s)$ where $E$ is a constant endofunctor. 
\end{remark}
\section{Completeness and behaviour in \torpdf{$\xmly$} and \torpdf{$\xmoo$} }\label{sec:completeness}
The first result that we want to generalise to $F$-machines is the well-known fact that, considering for example Mealy machines, if $(\clK,\otimes)$ has countable coproducts preserved by each $I\otimes \firstblank$, then the span \eqref{mmach_eq} can be `extended' to a span
\[\mealy{E}{d^+}{s^+}[I^+][O]\label{mly_extension}\]
where $d^+,s^+$ can be defined inductively from components $d_n,s_n : E\otimes I^{\otimes n} \to E,O$.

Under the same assumptions, each Moore machine \eqref{momach_eq} can be `extended' to a span
\[\mealy{E}{d^*}{s^*}[I^*][O]\label{mre_extension}\]
where $d^*,s^*$ can be defined inductively from components $d_n,s_n : E\otimes I^{\otimes n} \to E,O$.\footnote{Assuming countable coproducts in $\clK$, the free \emph{monoid} $I^*$ on $I$ is the object $\sum_{n\ge 0} I^n$; the free \emph{semigroup} $I^+$ on $I$ is the object $\sum_{\ge 1} I^n$; clearly, if $1$ is the monoidal unit of $\otimes$, $I^*\cong 1+I^+$, and the two objects satisfy `recurrence equations' $I^+\cong I\otimes I^+$ and $I^*\cong 1+I\otimes I^*$.}
\begin{remark}
    In the case of Mealy machines, the object $I^+$ corresponds to the \emph{free semigroup} on the input object $I$, whereas for Moore machines one needs to consider the \emph{free monoid} $I^*$: this mirrors the intuition that in the latter case an output can be provided without any previous input. Note that the extension of a Moore machine gives rise to a span of morphisms from the same object $E \otimes I^*$, i.e., a Mealy machine that accepts the empty string as input.
\end{remark}

A similar construction can be carried over in the category of $F$-Mealy machines, using the $F$-algebra map $d : FE\to E$ to generate iterates $E\xot{d_n} F^nE\xto{s_n} O$:

From now on, let $F$ be an endofunctor of a category $\clK$ that has a right adjoint $R$. Examples of such arise naturally from the situation where a triple of adjoints $L\dashv G\dashv R$ is given, since we obtain adjunctions $LG\dashv RG$ and $GL\dashv GR$:
\begin{itemize}
\item every homomorphism of rings $f : A\to B$ induces a triple of adjoint functor between the categories of $A$ and $B$-modules;
\item similarly, every homomorphism of monoids $f : M\to N$ induces a `base change' functor $f^* : N\emdash\Set\to M\emdash\Set$;
\item every essential geometric morphism between topoi $\clE\leftrightarrows\clF$, i.e. every triple of adjoints $f_!\dashv f^*\dashv f_*$;
\item every topological functor $V : \clE \to \clB$ \cite[Prop. 7.3.7]{Bor2} with its fully faithful left and right adjoints $L\dashv V\dashv R$ (this gives rise to a comodality $LV$, left adjoint to a modality $RV$).
\end{itemize}
\begin{construction}[Dynamics of an $F$-machine]\label{nth_skip}\agdasymbolraw{FMoore/Limits.agda\#L44} 
	For any given $F$-Mealy machine
	\[\xmach{E}{s}{d}\]
	we define the family of morphisms $s_n : F^nE\to O$ inductively, as the composites
	\[\label{sheeft}
		\begin{cases}
			s_1 & = FE \xto{s} O                                                                \\
			s_2 & = FFE \xto{Fd} FE \xto s O                                                    \\
			s_n & = F^n E \xto{F^{n-1} d} F^{n-1}E \to \cdots \xto{FFd} FFE\xto{Fd} FE \xto s O
		\end{cases}
	\]
	Under our assumption that $F$ has a right adjoint $R$, this is equivalent to the datum of their mates $\bar s_n : E\to R^n O$ for $n\ge 1$ 
	under the adjunction $F^n\adjunct{}{\eta_n} R^n$ obtained by composition, iterating the structure in $F\adjunct{\epsilon}{\eta} R$.

	Such a $s_n$ is called the $n$th \emph{skip map}. Observe that the datum of the family of all $n$th skip maps $(s_n\mid n\in\bbN_{\ge 1})$ is obviously equivalent to a single map of type $\bar s_\infty : E \to \prod_{n\ge 1} R^nO$.
\end{construction}
\begin{remark}
	Reasoning in a similar fashion, one can define extensions $s : E\to O$, $s\circ d : FE \to E\to O$, $s\circ d\circ Fd : FFE\to O$, etc.\@ for an $F$-Moore machine.
\end{remark}
This is the first step towards the following statement, which will be substantiated and expanded in \autoref{limcolim} below:
\begin{claim}\label{the_terminal}
	The category $\xmoo$ of \autoref{xmoore} has a terminal object $\fko=(O_\infty,d_\infty,s_\infty)$ with carrier $O_\infty = \prod_{n\ge 0} R^nO$; similarly, the category $\xmly$ has a terminal object with carrier $O_\infty = \prod_{n\ge 1} R^nO$. (Note the shift in the index of the product, motivated by the fact that the skip maps for a Moore machine are indexed on $\bbN_{\ge 0}$, and on $\bbN_{\ge 1}$ for Mealy.)
\end{claim}
The `modern' way to determine the presence of a terminal object in categories of automata relies on the elegant coalgebraic methods in \cite{jacobs_2016}; the interest in such completeness theorems can be motivated essentially in two ways:
\begin{itemize}
	\item the terminal object $O_\infty$ in a category of machines tends to be `big and complex', as a consequence of the fact that it is often a terminal coalgebra for a suitably defined endofunctor of $\clK$, so Adamek's theorem presents it as inverse limit of an op-chain.
	\item Coalgebra theory allows us to define a \emph{bisimulation} relation between states of different $F$-algebras (or, what is equivalent in our blanket assumptions, $R$-coalgebras), which in the case of standard Mealy/Moore machines (i.e., when $F=\firstblank\otimes I$) recovers the notion of bisimulation expounded in \cite[Ch. 3]{jacobs_2016}.
\end{itemize}
The following universal characterisation of both categories as pullbacks in $\Cat$ allows us to reduce the whole problem of completeness to the computation of a terminal object, and thus prove \autoref{limcolim}.
\begin{proposition}
\label{machines_are_pb}\leavevmode
	\begin{enumtag}{cx}
		\item the category $\xmly$ of $F$-Mealy machines given in \autoref{xmealy} fits in a pullback square
		\[\label{mly_pb}\vxy{
				\xmly \ar[r]^{U'}\ar[d]_{V'} \xpb& (F_{/O})\ar[d]^V \\
				\cate{Alg}(F) \ar[r]_U & \clK
			}\]
		\item the category $\xmly$ of $F$-Moore machines given in \autoref{xmoore} fits in a pullback square
		\[\label{mre_pb}\vxy{
				\xmoo\ar[r]^{U'}\ar[d]_{V'} \xpb & (\clK_{/O})\ar[d]^V \\
				\cate{Alg}(F) \ar[r]_U & \clK
			}\]
	\end{enumtag}
\end{proposition}
\begin{proof}
	Straightforward inspection of the definition of both pullbacks.
\end{proof}
As a consequence of this characterization, by applying \cite[V.6, Ex. 3]{McL} we can easily show the following completeness result:
\begin{theorem}[Limits and colimits of $F$-machines]\label{limcolim}\agdasymbol{FMoore/Limits}\leavevmode
	\begin{itemize}
		\item Let $\clK$ be a category admitting colimits of shape $\clJ$; then, $\xmoo$ and $\xmly$ have colimits of shape $\clJ$, and they are computed as in $\clK$;
		\item Equalizers (and more generally, all connected limits) are computed in $\xmoo$ and $\xmly$ as they are computed in $\clK$; if $\clK$ has countable products and pullbacks, $\xmoo$ and $\xmly$ also have products of any finite cardinality (in particular, a terminal object).
	\end{itemize}
\end{theorem}
\begin{proof}[Proof of \autoref{limcolim}]\label{proof_of_limcolim}
	It is worth unraveling the content of \cite[V.6, Ex. 3]{McL}, from which the claim gets enormously simplified: the theorem asserts that in any strict pullback square of categories
	\[\vxy{
			\clA \ar[r]^{U'}\xpb\ar[d]_{V'}& \clB \ar[d]^V\\
			\clC \ar[r]_U & \clK
		}\]
	if $U$ creates, and $V$ preserves, limits of a given shape $\clJ$, then $U'$ creates limits of shape $\clJ$. Thus, thanks to \autoref{machines_are_pb}, all connected limits (in particular, equalizers) are created in the categories of $F$-Mealy and $F$-Moore machines by the functors $U' : \xmly \to (F_{/O})$ and are thus computed as in $(F_{/O})$, i.e. as in $\clK$; this result is discussed at length in \cite[Ch. 4]{Ehrig} in the case of $(\firstblank\otimes I)$-machines, i.e. classical Mealy machines, to prove the following:
	\begin{itemize}
		\item assuming $\clK$ is cocomplete, all colimits are computed in $\xmly$ as they are computed in the base $\clK$;
		\item assuming $\clK$ has connected limits, they are computed in $\xmly$ as they are computed in the base $\clK$;
	\end{itemize}
	Discrete limits have to be treated with additional care: for classical Moore machines (cf. \autoref{moore_1cells}) the terminal object is the terminal coalgebra of the functor $A\mapsto A^I\times O$ (cf. \cite[2.3.5]{jacobs_2016}): a swift application of Adàmek theorem yields the object $[I^*,O]$; for classical Mealy machines (cf. \autoref{mmach_1cells}) the terminal object is the terminal coalgebra for $A\mapsto [I,O]\times [I,A]$; similarly, Adàmek theorem yields $[I^+,O]$.

	Adàmek theorem then yields the terminal object of $\xmoo$ as the terminal coalgebra for the functor $A\mapsto O\times RA$, which is the $O_{\infty,0}$ of \autoref{the_terminal}, and the terminal object of $\xmly$ as $O_{\infty,1}$ and for $A\mapsto RO\times RA$ (in $\xmly$).
	All discrete limits can be computed when pullbacks and a terminal object have been found, but we prefer to offer a more direct argument to build binary products.

	Recall from \autoref{nth_skip} the definition of dynamics map associated to an $F$-machine $\fke=(E,d,s)$.

	Now, our claim is two-fold:
	\begin{enumtag}{to}
		\item \label{to_1} the object $O_\infty := \prod_{n\ge 1} R^nO$ in $\clK$ carries a canonical structure of an $F$-machine $\fko=(O_\infty,d_\infty,s_\infty)$ such that $\fko$ is terminal in $\xmly$;
		\item \label{to_2} given objects $(E,d_E,s_E),(T,d_T,s_T)$ of $\xmly$, the pullback
		\[\label{its_a_pb}\vxy{
			P_\infty\ar[r]\ar[d]\xpb & T \ar[d]^{\bar s_{T,\infty}}\\
			E \ar[r]_{\bar s_{E,\infty}} & O_\infty
			}\]
		is the carrier of a $F$-machine structure that exhibits $\fkp=(P_\infty, d_P,s_P)$ as the product of $\fke=(E,d_E,s_E),\fkf=(T,d_T,s_T)$ in $\xmly$.
	\end{enumtag}
	In this way, the category $\xmly$ comes equipped with all finite products; is easy to prove a similar statement when an infinite number of objects $(\fke_i\mid i\in I)$ is given by using wide pullbacks whenever they exist in the base category.

	Observe that the object $P_\infty$ can be equivalently characterized as the single wide pullback obtained from the pullback $P_n$ of $\bar s_{E,n}$ and $\bar s_{T,n}$ (or rather, an intersection, since each $P_n\to E\times T$ obtained from the same pullback is a monomorphism):
	\[\vxy{
		P_\infty\ar[r]\ar[d] \ar@{}[dr]|{\rotatebox[origin=c]{90}{$\ddots$}} & P_n \ar[d]\\
		P_m \ar[r]& E\times T
		}\label{another_pinfty}\]

	Showing the universal property of $P_\infty$ will be more convenient at different times in one or the other definition.

	In order to show our first claim in \ref{to_1}, we have to provide the $F$-machine structure on $O_\infty$, exhibiting a span
	\[\xmach{O_\infty}{s_\infty}{d_\infty}\]
	on one side, $s_\infty$ is the adjoint map of the projection $\pi_1 : O_\infty\to RO$ on the first factor; the other leg $d_\infty$ is the adjoint map of the projection deleting the first factor, thanks to the identification $RO_\infty\cong \prod_{n\ge 2}R^nO$; explicitly then, we are considering the following diagram:
	\[\vxy{
		O_\infty &\ar[l]_{\epsilon_{O_\infty}}FR O_\infty & \ar[l]_{F\pi_{\ge 2}}F O_\infty \ar[r]^{F\pi_1}& FR O \ar[r]^-{\epsilon_O} & O
		}\]
	To prove the first claim, let's consider a generic object $(E,d,s)$ of $\xmly$, i.e. a span
	\[\xmach{E}{s}{d}\]
	and let's build a commutative diagram
	\[\vxy{
		E \ar[d]_u&& FE\ar[rr]^s\ar[d]^{Fu}\ar[ll]_d  && O\ar@{=}[d]\\
		O_\infty &\ar[l]_{\epsilon_{O_\infty}}FR O_\infty & \ar[l]_{F\pi_{\ge 2}}F O_\infty \ar[r]^{F\pi_1}& FR O \ar[r]^-{\epsilon_O} & O
		}\label{terminality_of_O}\]
	for a unique morphism $u : E\to O_\infty=\prod_{n\ge 1} R^nO$ that we take exactly equal to $\bar s_\infty$. The argument that $u$ makes diagram \eqref{terminality_of_O} commutative, and that it is unique with this property, is now a completely straightforward diagram chasing.

	Now let's turn to the proof that the tip of the pullback in \eqref{its_a_pb} exhibits the product of $(E,d_E,s_E),(T,d_T,s_T)$ in $\xmly$; first, we build the structure morphisms $s_P,d_P$ as follows:
	\begin{itemize}
		\item $d_P$ is the dotted map obtained thanks to the universal property of $P_\infty$ from the commutative diagram
		      \[\vxy[@C=4mm@R=4mm]{
				      FP_\infty\ar@{.>}[dr]^{d_P} \ar@{->}[rr] \ar@{->}[dd] &  & FE \ar@{->}[dd]|\hole \ar@{->}[rd] &  \\
				      & P_\infty \ar@{->}[rr] \ar@{->}[dd] &  & E \ar@{->}[dd] \\
				      FT \ar@{->}[rr]|\hole \ar@{->}[rd] &  & FO_\infty \ar@{->}[rd] &  \\
				      & T \ar@{->}[rr] &  & O_\infty
			      }\]
		\item $s_P : FP_\infty\to O$ is obtained as the adjoint map of the diagonal map $P_\infty\to O_\infty$ in \eqref{its_a_pb} composed with the projection $\pi_1 : O_\infty \to RO$.
	\end{itemize}
	Let's now assess the universal property of the object
	\[\xmach{P_\infty}{s_P}{d_P}\]
	We are given an object $\fkz =(Z,d_Z,s_Z)$ of $\xmly$ and a diagram
	\[
		\vxy{
			O \ar@{=}[r] & O & O \ar@{=}[l] \\
			FE \ar@{->}[d]_{d_E} \ar@{->}[u]^{s_E} & FZ \ar@{->}[u]^{s_Z} \ar@{->}[d]^{d_Z} \ar@{->}[r]_{Fv} \ar@{->}[l]^{Fu} & FT \ar@{->}[u]_{s_T} \ar@{->}[d]^{d_T} \\
			E & Z \ar@{->}[r]_{v} \ar@{->}[l]^{u} & T
		}
	\]
	commutative in all its parts. To show that there exists a unique arrow $[u,v] : Z\to P_\infty$
	\[\xymatrix{
		& Z \ar@{.>}[d]|{[u,v]} \ar@{->}[ld]_{u} \ar@{->}[rd]^{v} &  \\
		E & P_\infty \ar@{->}[l]^{p_E} \ar@{->}[r]_{p_T} & T
		}
	\]
	we can argue as follows, using the joint injectivity of the projection maps $\pi_n : O_\infty\to R^nO$: first, we show that each square
	\[\vxy{
		Z \ar@{->}[r]^{u} \ar@{->}[d]_{v} & E \ar@{->}[d]^{{\bar s_{E,n}}} \\
		T \ar@{->}[r]_{{\bar s_{T,n}}} & R^nO
		}\]
	is commutative, and in particular that its diagonal is equal to the $n$th skip map of $Z$; this can be done by induction, showing that the composition of both edges of the square with the canonical projection $O_\infty \to R^nO$ equals $\bar s_{n,Z}$ for all $n\ge 1$. From this, we deduce that there exist maps
	\[\vxy{Z \ar[r]^{z_n} & P_n \ar[r]& E\times T}\]
	(cf. \eqref{another_pinfty} for the definition of $P_n$) for every $n\ge 1$,
	But now, the very way in which the $z_n$s are defined yields that each such map coincides with $\langle u,v\rangle : Z\to E\times T$, thus $Z$ must factor through $P_\infty$. Now we have to exhibit the commutativity of diagrams
	\[\vxy{
		Z \ar@{->}[d]_{{[u,v]}} & FZ \ar@{->}[d]^{{F[u,v]}} \ar@{->}[l]_{d_Z} \ar@{->}[r]^{s_Z} & O \ar@{=}[d] \\
		P_\infty & FP_\infty \ar@{->}[l]^{d_P} \ar@{->}[r]_{s_P} & O
		}\]
	and this follows from a straightforward diagram chasing.

	This concludes the proof.
\end{proof}
\begin{remark}
	Phrased out explicitly, the statement that $\fko=(O_\infty,d_\infty,s_\infty)$ is a terminal object amounts to the fact that given any other $F$-Mealy machine $\fke=(E,d,s)$, there is a unique $u_E : E\to O_\infty$ with the property that
	\[\vxy{
			E \ar[d]_{u_E}& FE\ar[r]^s\ar[d]^{Fu_E}\ar[l]_d  & O\ar@{=}[d]\\
			O_\infty & \ar[l]^{d_\infty} F O_\infty \ar[r]_{s_\infty} & O
		}\]
	are both commutative diagrams; a similar statement holds for $F$-Moore automata.
\end{remark}
\subsection{Adjoints to behaviour functors}\label{roba_naude}
In \cite{2d281e1e3e0b525128f55519cf8a03fa52ce6252,NAUDE1979277} the author concentrates on building an adjunction between a category of machines and a category collecting the \emph{behaviours} of said machines.


Call an endofunctor $F:\clK\to\clK$ an \emph{input process} if the forgetful functor $U : \Alg(F)\to\clK$ has a left adjoint $G$; in simple terms, an input process allows to define free $F$-algebras.\footnote{Obviously, this is in stark difference with the requirement that $F$ has an adjoint, and the two requests are independent: if $F$ is a monad, it is always an input process, regardless of $F$ admitting an adjoint on either side.}

In \cite{2d281e1e3e0b525128f55519cf8a03fa52ce6252,NAUDE1979277} the author concentrates on proving the existence of an adjunction
\[\vxy{L : \cate{Beh}(F) \ar@<.33pc>[r]\ar@{}[r]|\perp&\ar@<.33pc>[l] \cate{Mach}(F) : E}\]
where $\cate{Mach}(F)$ is the category obtained from the pullback
\[\label{naude_diag}\vxy{
		\cate{Mach}(F) \xpb \ar[rr]\ar[d]& & \clK^\to\times\clK^\to \ar[d]^{d_1\times d_0}\\
		\cate{Alg}(F) \ar[r]_{U}& \clK \ar[r]_\Delta & \clK\times\clK
	}\]
$\Delta$ is the diagonal functor, and $\cate{Beh}(F)$ is a certain comma category on the free $F$-algebra functor $G$.

Phrased in this way, the statement is conceptual enough to carry over to $F$-Mealy and $F$-Moore machines (and by extension, to all settings where a category of automata can be presented through a strict 2-pullback in $\Cat$ of well-behaved functors --a situation that given \eqref{mly_pb}, \eqref{mre_pb}, \eqref{naude_diag} arises quite frequently).
\begin{theorem}\label{naude_thm}\agdasymbol{FMoore/Adjoints} 
	There exists a functor $B : \xmoo \to \cate{Alg}(F)_{/(O_\infty,d_\infty)}$, where the codomain is the slice category of $F$-algebras and the $F$-algebra $(O_\infty,d_\infty)$ is determined in \autoref{the_terminal}. The functor $B$ has a left adjoint $L$.
\end{theorem}
\begin{proof}\label{proof_of_naude_thm}
	Recall that the functor $B :\xmoo_O\to \cate{Alg}(F)_{/O_\infty}$ is defined as follows on objects and morphisms:
	\[\vxy[@R=5mm]{
			(E,d,s) \ar@{->}[dd]^{f} &  & E \ar@{->}[rd]^{u_E} \ar@{->}[dd]_{f} &  \\
			& \mapsto &  & O_\infty \\
			(T,d',s') &  & T \ar@{->}[ru]_{u_T} &
		}\]
	A typical object of $\cate{Alg}(F)_{/O_\infty}$ is a tuple $((A,a),u)$ where $a : FA\to A$ is an $F$-algebra with its structure map, and $u : A\to O_\infty$ is an $F$-algebra homomorphism, i.e. a morphism $u$ such that $d_T\circ Fu = u\circ a$.

	A putative left adjoint for $B$ realises a natural bijection
	\[\xmoo_O\big(L((A,a), u), (E,d,s)\big)\cong \cate{Alg}(F)_{/O_\infty}\big(((A,a),u),B(E,d,s)\big)\]
	between the following two kinds of commutative diagrams:
	\[\begin{tikzpicture}[baseline=(current bounding box.center)]
			\node[font=\tiny] (dis) at (0,0) {\uno};
			\node[font=\tiny] (dat) at (7,0) {\tre};
			\path[->] (dis) edge[shift left=2mm] (dat);
			\path[<-] (dis) edge[shift left=-2mm] (dat);
		\end{tikzpicture}\]
	There is a clear way to establish this correspondence.
\end{proof}
The functor $B$ is defined as follows:
\begin{itemize}
	\item on objects $\fke=(E,d,s)$ in $\xmoo$, as the correspondence sending $\fke$ to the unique map $u_E : E\to O_\infty$, which is an $F$-algebra homomorphism by the construction in \eqref{terminality_of_O};
	\item on morphisms, $f : (E,d,s)\to (F,d',s')$ between $F$-Moore machines, $B$ acts as the identity, ultimately as a consequence of the fact that the terminality of $O_\infty$ yields at once that $u_F\circ f=u_E$.
\end{itemize}
\begin{remark}\agdasymbol{FMoore/Adjoints.agda\#L194}
	The adjunction in \autoref{naude_thm} is actually part of a longer chain of adjoints obtained as follows: recall that every adjunction $G : \clK \leftrightarrows \clH :  U$ induces a `local' adjunction $\tilde G : \clK_{/UA} \leftrightarrows \clH_{/A} : \tilde U$ where $\tilde U(FA,f : FA\to A) = Uf$. Then, if $F$ is an input process, we get adjunctions
	\[\xymatrix{		\clK_{/O_\infty} \ar@<.33pc>[r]^-{\tilde G}\ar@{}[r]|-\perp&\ar@<.33pc>[l]^-{\tilde U}
		\cate{Alg}(F)_{/(O_\infty,d_\infty)} \ar@<.33pc>[r]^-{L}\ar@{}[r]|-\perp&\ar@<.33pc>[l]^-B
		\xmoo.}\]
\end{remark}

\bibliography{refs}
\end{document}

%% file: cmds.tex
\usepackage{quiver}
\usepackage[normalem]{ulem}
\newcommand{\torpdf}[1]{\texorpdfstring{#1}{}}

\def\baseRepo{https://github.com/iwilare/categorical-automata/blob/main/}
\newcommand{\Href}[2]{\href{#1}{\textsf{#2}}}

\newcommand{\agdasymbol}[1]{%
	\hspace{0.07em}%
	\Href{\baseRepo #1.agda}{{\upshape(\raisebox{-0.2em}{%
			\includegraphics[height=1em]{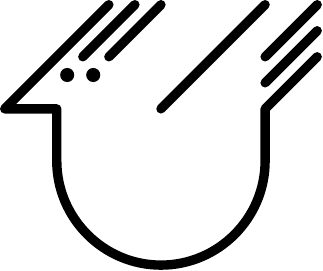}%
		})}%
		\hspace{-0.40em}%
  }
}
\newcommand{\agdasymbolraw}[1]{%
	\hspace{0.07em}%
	\Href{\baseRepo #1}{{\upshape(\raisebox{-0.2em}{%
			\includegraphics[height=1em]{agda}%
		})}%
		\hspace{-0.40em}%
  }
}
\newcommand{\specialagdasymbol}{%
	\hspace{0.07em}%
	\Href{https://github.com/iwilare/categorical-automata}{{\upshape(\raisebox{-0.2em}{%
			\includegraphics[height=1em]{agda}%
		})}%
		\hspace{-0.40em}%
  }
}
\NewDocumentCommand{\vxy}{o m}{
	\IfNoValueTF{#1}
	{\vcenter{\xymatrix{#2}}}
	{\vcenter{\xymatrix#1{#2}}}
}

\NewDocumentCommand{\mealy}{m m m O{I} O{O}}{\vxy{#1 & \ar[l]_-{#2} #1\otimes #4 \ar[r]^-{#3} & #5}}
\NewDocumentCommand{\moore}{m m m O{I} O{O}}{\vxy{#1 & \ar[l]_-{#2} #1\otimes #4 \:\: ; \:\: #1 \ar[r]^-{#3} & #5}}

\def\Mly{\cate{Mly}}

\def\mly{\cate{Mly}}

\def\Mre{\cate{Mre}}
\let\mre\Mre

\usetikzlibrary{decorations}
\pgfdeclaredecoration{simple line}{initial}{
	\state{initial}[width=\pgfdecoratedpathlength-1sp]{\pgfmoveto{\pgfpointorigin}}
	\state{final}{\pgflineto{\pgfpointorigin}}
}
\tikzset{
	shift left/.style={decorate,decoration={simple line,raise=#1}},
	shift right/.style={decorate,decoration={simple line,raise=-1*#1}},
}

\def\concat{\mathbin{\scalebox{.8}{$+\kern-.2em+$}}}

%% file: boilerplate.tex
\EventEditors{John Q. Open and Joan R. Access}
\EventNoEds{2}
\EventLongTitle{42nd Conference on Very Important Topics (CVIT 2016)}
\EventShortTitle{CVIT 2016}
\EventAcronym{CVIT}
\EventYear{2016}
\EventDate{December 24--27, 2016}
\EventLocation{Little Whinging, United Kingdom}
\EventLogo{}
\SeriesVolume{42}
\ArticleNo{23}